\documentclass[11pt,twoside, final]{amsart}

\copyrightinfo{0}{Iranian Mathematical Society}

\usepackage{amsmath,amsthm,amscd,amsfonts,amssymb,enumerate}
\usepackage{enumerate}
\usepackage{mathrsfs,euscript}
\usepackage{graphicx}		
\usepackage{color}
\usepackage[colorlinks]{hyperref}

\newtheorem{theorem}{Theorem}[section]
\newtheorem{proposition}[theorem]{Proposition}
\newtheorem{lemma}[theorem]{Lemma}
\newtheorem{corollary}[theorem]{Corollary}

\theoremstyle{definition}
\newtheorem{definition}[theorem]{Definition}
\newtheorem{example}[theorem]{Example}

\theoremstyle{remark}

\numberwithin{equation}{section}

\newcommand{\e}{\varepsilon}

\DeclareMathOperator{\mis}{\mathfrak{M}}
\DeclareMathOperator{\Lip}{Lip}
\DeclareMathOperator{\Inv}{Inv}

\newcommand{\M}{\mathscr{M}}
\newcommand{\m}{\mathsf{M}}
\newcommand{\J}{\EuScript{J}}

\newcommand{\C}{\mathbb{C}}
\newcommand{\N}{\mathbb{N}}
\newcommand{\U}{\mathbf{1}}
\newcommand{\z}{\mathbf{0}}

\newcommand{\pp}{$\mathscr{P}$}
\newcommand{\tp}{\otimes}
\newcommand{\ptp}{\mathbin{\wh{\otimes}}}
\newcommand{\wh}{\widehat}

\newcommand{\A}{\mathscr{A}}
\newcommand{\fA}{\mathfrak{A}}

\newcommand{\cI}{\mathcal{I}}
\newcommand{\cX}{\mathcal{X}}

\newcommand{\es}{\emptyset}

\newcommand{\set}[1]{\{#1\}}

\newcommand{\biggset}[1]{\biggl\{ #1 \biggr\}}

\newcommand{\Bigprn}[1]{\Bigl( #1 \Bigr)}

\newcommand{\enorm}{\lVert\,\cdot\,\rVert}
\newcommand{\norm}[1]{\lVert #1 \rVert}
\newcommand{\Bignorm}[1]{\Bigl\lVert #1 \Bigr\lVert}
\newcommand{\lVertt}{\lvert\mspace{-2mu}\lvert\mspace{-2mu}\lvert}
\newcommand{\rVertt}{\rvert\mspace{-2mu}\rvert\mspace{-2mu}\rvert}
\newcommand{\enormm}{\lVertt\, \cdot \,\rVertt}
\newcommand{\normm}[1]{\lVertt #1 \rVertt}

\commby{Hamid Pezeshk}

\begin{document}


\title[The character space of vector-valued function algebras]%
{On the character space of Banach vector-valued function algebras}

\author[M. Abtahi]{Mortaza Abtahi}
\address{Mortaza Abtahi\newline
School of Mathematics and Computer Sciences,\newline
Damghan University, Damghan, P.O.BOX 36715-364, Iran}

\email{abtahi@du.ac.ir; mortaza.abtahi@gmail.com}



\maketitle

\begin{abstract}
  Given a compact space $X$ and a commutative Banach algebra
  $A$, the character spaces of $A$-valued function algebras on $X$ are
  investigated. The class of natural $A$-valued function algebras,
  those whose characters can be described by means of characters of $A$ and
  point evaluation homomorphisms, is introduced and studied. For an
  admissible Banach $A$-valued function algebra $\A$ on $X$,
  conditions under which the character space $\mis(\A)$
  is homeomorphic to $\mis(\fA)\times \mis(A)$ are presented, where $\fA=C(X) \cap\A$
  is the subalgebra of $\A$ consisting of scalar-valued functions.
  An illustration of the results is given by some examples.\\[1ex]
\end{abstract}

\section{\bf Introduction and Preliminaries}
\label{sec:intro}

Let $A$ be a commutative unital Banach algebra over the complex field $\C$.
Every nonzero homomorphism $\phi:A\to\C$
is called a \emph{character} of $A$. Denoted by $\mis(A)$, the set of
all characters of $A$ is nonempty and its elements are automatically continuous
\cite[Lemma 2.1.5]{CBA}. Consider the Gelfand transform $\hat A=\set{\hat a:a\in A}$,
where $\hat a:\mis(A)\to\C$ is defined by $\hat a(\phi)=\phi(a)$.
The \emph{Gelfand topology} of $\mis(A)$ is the weakest topology
with respect to which every $\hat a\in \hat A$ is continuous.
Endowed with the Gelfand topology,
$\mis(A)$ is compact and Hausdorff. By \cite[Theorem 2.1.8]{CBA}, an ideal $M$ in $A$
is maximal if and only if $M=\ker\phi$, for some $\phi\in\mis(A)$. For this reason,
sometimes $\mis(A)$ is called the \emph{maximal ideal space} of $A$.
For more on the theory of commutative Banach algebras see, for example,
\cite{BD,Dales,CBA,Zelazko}.

\subsection{Function Algebras}
Let $X$ be a compact Hausdorff space and $C(X)$ be the Banach algebra of all continuous
functions $f:X\to\C$ equipped with the uniform norm $\|f\|_X=\sup\set{|f(x)|:x\in X}$.
A subalgebra $\fA$ of $C(X)$ is called a \emph{function algebra} on $X$ if $\fA$
separates the points of $X$ and contains the constant functions. If $\fA$ is equipped with
some complete algebra norm $\enorm$,
then $\fA$ is called a \emph{Banach function algebra}. If the norm $\enorm$ of
$\fA$ is equivalent to the uniform norm $\enorm_X$, then $\fA$ is called
a \emph{uniform algebra}.

Identifying the character space of a Banach function algebra $\fA$ has
been always a problem of interest for mathematicians in this field.
For every $x\in X$, the evaluation homomorphism
$\e_x:f\mapsto f(x)$ is a character of $\fA$, and
the mapping $J:X\to\mis(\fA)$, $x\mapsto\e_x$, imbeds $X$ homeomorphically as
a compact subset of $\mis(\fA)$. If $J$ is surjective, one calls $\fA$
\emph{natural} \cite[Chapter 4]{Dales}. In this case, the character space $\mis(\fA)$
is identical to $X$. Note that every semisimple commutative Banach algebra
$A$ can be considered, through its Gelfand representation, as a natural Banach function
algebra on its character space $\mis(A)$.

A relation between the character space $\mis(\fA)$ of a Banach function algebra $\fA$
and the character space $\mis(\bar\fA)$ of its uniform closure $\bar\fA$ was revealed
in \cite{Honary} as follows.

\begin{theorem}[Honary \cite{Honary}]
\label{thm:honary}
  The restriction map $\mis(\bar\fA)\to \mis(\fA)$, $\psi\mapsto \psi|_\fA$,
  is a homeomorphism if and only if $\|\hat f\| \leq \|f\|_X$, for all $f\in\fA$.
\end{theorem}

The above result appears to be very useful in identifying the character spaces in
a wide class of Banach function algebras.  We establish an analogue of this result
for vector-valued function algebras in Section \ref{sec:characters}.

\subsection{Vector-valued Function Algebras}
Let $A$ be a commutative unital Banach algebra, and let $C(X,A)$ be the space of
all $A$-valued continuous functions on $X$. Algebraic operations and the uniform norm $\enorm_X$
on $C(X,A)$ are defined in the obvious way.

\begin{definition}[c.f. \cite{Abtahi-vector-valued,Nikou-Ofarrell}]
\label{dfn:vector-valued-function-algebras}
  A subalgebra $\A$ of $C(X,A)$ is called an \emph{$A$-valued function algebra} on $X$
  if (1) $\A$ contains the constant functions $X\to A$, $x\mapsto a$,
  for all $a\in A$, and (2) $\A$ separates the points of $X$
  in the sense that, for every pair $x,y\in X$ with $x\neq y$, and
  for every maximal ideal $M$ of $A$, there exists some $f \in A$ such that
  $f(x)-f(y)\notin M$. If $\A$ is endowed with some algebra norm $\enormm$ such that
  the restriction of $\enormm$ to $A$ is equivalent to the original norm of $A$
  and $\|f\|_X \leq \normm{f}$, for every $f\in \A$, then $\A$ is called
  a \emph{normed $A$-valued function algebra} on $X$. If the given norm is
  complete, then $\A$ is called a \emph{Banach $A$-valued function algebra}.
  If the given norm is equivalent to the uniform norm $\enorm_X$,
  then $\A$ is called an \emph{$A$-valued uniform algebra}.
  When no confusion can arise, we use the same notation $\enorm$ for the norm
  of $\A$.
\end{definition}

Continuing the work of Yood \cite{Yood}, Hausner \cite{Hausner} proved that
$\tau$ is a character of $C(X,A)$ if, and only if,
there exist a point $x\in X$ and a character $\phi\in\mis(A)$ such that
$\tau(f) = \phi(f(x))$, for all $f\in C(X,A)$, whence $\mis(C(X,A))$ is
homeomorphic to $X \times \mis(A)$. (In this regard, see \cite{Ab-Ab}.)
We call a Banach $A$-valued function algebra \emph{natural} if, like $C(X,A)$,
its character space is identical to $X\times \mis(A)$. For instance,
in Example \ref{exa:Lip(X,A)}, we will see that
the $A$-valued Lipschitz algebra $\Lip(X,A)$ is natural;
see also \cite{Esmaeili-Mahyar}, \cite{Honary-Nikou-Sanatpour}.
Natural $A$-valued function algebras are studied
in Section \ref{sec:natural-function-algebras}.

\subsection{Notations and conventions}
  Throughout, $X$ is a compact Hausdorff space, and
  $A$ is a \emph{semisimple} commutative unital Banach algebra. The unit element of $A$
  is denoted by $\U$, and the set of invertible elements of $A$ is denoted by $\Inv(A)$.
  If $f:X\to\C$ is a function and $a\in A$, we write $fa$ to denote the $A$-valued function
  $X\to A$, $x\mapsto f(x)a$. If $\fA$ is a function algebra on $X$, we let $\fA A$ be
  the linear span of $\set{fa:f\in\fA,\,a\in A}$, so that any element $f\in \fA A$
  is of the form $f=f_1a_1+\dotsb+f_na_n$ with $f_j\in \fA$ and $a_j\in A$.
  Given an element $a\in A$, we use the same notation $a$ for the constant
  function $X\to A$ given by $a(x)=a$, for all $x\in X$, and consider $A$ as a closed
  subalgebra of $C(X,A)$. Since $A$ has a unit element $\U$,
  we identify $\C$ with the closed subalgebra $\C\U$ of $A$. Whence every continuous
  function $f:X\to\C$ can be considered as the continuous $A$-valued function
  $f\U:x\mapsto f(x)\U$. We drop $\U$ using the same notation $f$ for
  this $A$-valued function and adopt the identification $C(X)=C(X)\U$
  as a closed subalgebra of $C(X,A)$. Finally, for a family $\M$ of $A$-valued functions
  on $X$, a point $x\in X$, and a character $\phi\in\mis(A)$, we set
  \[
    \M(x) = \set{f(x):f\in\M}, \quad
    \phi[\M]=\set{\phi\circ f: f\in\M}.
  \]

\section{Natural Vector-valued Function Algebras}
\label{sec:natural-function-algebras}

Let $\A$ be an $A$-valued function algebra on $X$. Assume that $M$ is a maximal ideal of $A$,
$x_0\in X$, and set
  \begin{equation}\label{eqn:natural-form}
    \M=\set{f\in\A:f(x_0)\in M}.
  \end{equation}
The fact that $\M$ is an ideal of $\A$ is obvious.
We prove that $\M$ is maximal. Take a function $g \in \A\setminus \M$
so that $g(x_0)\notin M$. Since $M$ is maximal in $A$,
there exist $a\in M$ and $b\in A$ such that
$\U = a + g(x_0)b$. Consider $b$ as a constant function of $X$ into $A$ and let $f=\U-gb$.
Then $f(x_0)=a\in M$ so that $f\in \M$ and $\U=f+gb$ which means that the ideal of $\A$
generated by $\M\cup\set{g}$ is equal to $\A$. Hence $\M$ is maximal in $\A$.

\begin{definition}
\label{dfn:natural-A-valued-FA}
  An $A$-valued function algebra $\A$ on $X$ is called
  \emph{natural} on $X$, if every maximal ideal $\M$ of $\A$ is of the form
  \eqref{eqn:natural-form}, for some $x_0\in X$ and $M\in \mis(A)$.
\end{definition}

In case $A=\C$, natural $A$-valued function algebras coincide with natural
(complex) function algebras.

\begin{theorem}
\label{thm:if M(x) neq A then}
  Let $\A$ be an $A$-valued function algebra on $X$.
  If $\M$ is a maximal ideal in $\A$ and $\M(x_0)\neq A$, for some
  $x_0\in X$, then
  \begin{enumerate}[\quad$(1)$]
    \item $\M(x_0)$ is a maximal ideal of $A$;

    \item $\M(x)=A$ for $x\neq x_0$;

    \item $\M=\set{f\in \A: f(x_0)\in M}$, where $M=\M(x_0)$.
  \end{enumerate}
\end{theorem}

\begin{proof}
  It is easily verified that $\M(x_0)$ is an ideal. We show that $\M(x_0)$ is maximal.
  Assume that $a\notin\M(x_0)$. Then $a$, as a constant function on $X$,
  does not belong to $\M$. Hence, the ideal of $\A$ generated by $\M\cup \set{a}$
  is equal to $\A$ meaning that $\U=f+ag$, for some $f\in\M$ and $g\in\A$. In particular,
  $\U=f(x_0)+ag(x_0)$ which implies that the ideal of $A$ generated by $\M(x_0)\cup\set{a}$
  is equal to $A$. Hence, $\M(x_0)$ is maximal.

  Now, assume that $x\neq x_0$.
  Since $\A$ separates the points of $X$ (Definition \ref{dfn:vector-valued-function-algebras}),
  for the maximal ideal $\M(x_0)$ in $A$, there is a function $f\in \A$ such that
  $f(x)-f(x_0)\notin \M(x_0)$. Define $g(s)=f(s)-f(x)$ so that
  $g(x_0)\notin \M(x_0)$. This implies that
  $g\notin\M$. Since $\M$ is maximal, there are $h\in \M$ and
  $k\in\A$ such that $h+kg=\U$. Hence, $\U=h(x)\in\M(x)$ and
  $\M(x)=A$.
\end{proof}

It is proved in \cite{Ab-Ab} that the algebra $C(X,A)$ satisfies all conditions
in Theorem \ref{thm:if M(x) neq A then}. Therefore $C(X,A)$ is natural.

\begin{corollary}
\label{cor:A-is-natural-iff-M(x)neqA}
  Let $\A$ be an $A$-valued function algebra on $X$.
  \begin{enumerate}[\quad$(1)$]
    \item The algebra $\A$ is natural if, and only if, for every proper ideal $\cI$ in $\A$,
    there exists some $x_0\in X$ such that $\cI(x_0)\neq A$.

    \item If $\cI$ is an ideal in $\A$ such that $\cI(x_0)$ and $\cI(x_1)$, for
  $x_0\neq x_1$, are proper ideals in $A$, then $\cI$ cannot be maximal in $\A$.
  \end{enumerate}
\end{corollary}

The next discussion requires a concept of zero sets. The zero set of a function $f:X\to A$
is defined as $Z(f)=\set{x:f(x)=\z}$. This concept of zero set, however,
is not useful here in our discussion because, in general, the algebra $A$ may contain
nonzero singular elements. Instead, the following slightly modified version of this concept
appears to be very useful.

\newcommand{\sing}{Z_\text{\upshape s}}

\begin{definition}
  For a function $f:X\to A$, the \emph{singular set} of $f$ is defined to be
  \begin{equation}
    \sing(f) = \set{x\in X: f(x) \notin \Inv(A)}.
  \end{equation}
\end{definition}


The following is an analogy of \cite[Proposition 4.1.5 (i)]{Dales}.

\begin{theorem}
\label{thm:A-is-natural-iff-f1-...-fn}
  Let $\A$ be a Banach $A$-valued function algebra on $X$.
  Then $\A$ is natural if, and only if,
  for each finite set $\set{f_1,\dotsc,f_n}$ of elements in $\A$
  with $\bigcap_{j=1}^n \sing(f_j)=\es$,
  there exist $g_1,\dotsc,g_n\in \A$ such that
  \[
    f_1g_1+\dotsb+f_ng_n=\U.
  \]
\end{theorem}

\begin{proof}
  $(\Rightarrow)$ Suppose that $\A$ is natural and, for a finite set $\set{f_1,\dotsc,f_n}$ in $\A$,
  assume that $\sing(f_{1})\cap \dotsb \cap \sing(f_{n})=\es$.
  Let $\cI$ be the ideal generated by $\set{f_1,\dotsc,f_n}$. If $\cI\neq\A$, then,
  since $\A$ is natural, by Corollary \ref{cor:A-is-natural-iff-M(x)neqA},
  there exists a point $x_0\in X$ such that $\cI(x_0)\neq A$. In particular,
  the elements $f_1(x_0),\dotsc, f_n(x_0)$ are all singular in $A$, which means
  that $x_0\in \sing(f_1) \cap \dotsb \cap \sing(f_n)$, a contradiction.
  Therefore, $\cI=\A$ whence there exist $g_1,\dotsc,g_n\in\A$ such that
  $f_1g_1+\dotsb+f_ng_n=\U$.

  $(\Leftarrow)$ To show that $\A$ is natural, we take a maximal ideal $\M$ of $\A$ and,
  using Corollary \ref{cor:A-is-natural-iff-M(x)neqA}, we show that $\M(x_0)\neq A$,
  for some $x_0\in X$. Assume, towards a contradiction, that, for every $x\in X$,
  there exists a function $f_x\in\M$ such that $f_x(x)=\U$.
  Set $V_x=f_x^{-1}(\Inv(A))$. Then $\set{V_x:x\in X}$ is an open covering of
  the compact space $X$. So there exist finitely many points $x_1,\dotsc,x_n\in X$ such that
  $X\subset V_{x_1}\cup \dotsb \cup V_{x_n}$.
  Then $\sing(f_{x_1})\cap \dotsb \cap \sing(f_{x_n})=\es$. By the assumption,
  there exist functions $g_1,\dotsc,g_n\in \A$ such that
  $f_{x_1}g_1+\dotsb+f_{x_n}g_n=\U$. Hence, $\U\in\M$, which is a contradiction.
\end{proof}

Let $f\in\A$ and suppose that $\sing(f)=\es$ so that $f(X)\subset\Inv(A)$.
Since the inverse mapping $a\mapsto a^{-1}$ of $\Inv(A)$ onto itself is
continuous, the mapping $x\mapsto f(x)^{-1}$, denoted by $\U/f$, is a continuous
$A$-valued function on $X$. Hence $f$ is invertible in $C(X,A)$.
However, $f$ may not be invertible in $\A$. Let us call
$\A$ a \emph{full subalgebra} of $C(X,A)$ if every $f\in \A$ that is invertible
in $C(X,A)$ is invertible in $\A$. The following is an analogy of
\cite[Theorem 2.1]{Abtahi-QM}.

\begin{theorem}
\label{thm:if-bar(A)-is-natural}
  Let $\A$ be a Banach $A$-valued function algebra on $X$ such that $\bar\A$, the
  uniform closure of $\A$, is natural. If $\U/f\in \A$ whenever $f\in \A$ and
  $\sing(f)=\es$, then $\A$ is natural.
\end{theorem}

\begin{proof}
  We apply Theorem \ref{thm:A-is-natural-iff-f1-...-fn} to prove that $\A$ is natural.
  Let $f_1,\dotsc, f_n$ be elements in $\A$ such that $\sing(f_1) \cap \dotsb \cap \sing(f_n)=\es$.
  We prove the existence of a finite set $\set{g_1,\dotsc, g_n}$ of elements in $\A$ such that
  $f_1g_1 + \dotsb +f_ng_n = \U$. Regarding $f_1, \ldots ,f_n$ as elements of $\bar \A$,
  since $\bar \A$ is natural, again by Theorem \ref{thm:A-is-natural-iff-f1-...-fn},
  there exist $h_1,\dotsc, h_n$ in $\bar \A$ such that $f_1 h_1+\cdots+f_nh_n =\U$.
  For each $h_j$, there is some $g_j\in \A$ such that $\|h_j-g_j\|_X <\sum_{j=1}^n \|f_j\|_X$. Thus
  \begin{equation}
    \Bignorm{\U-\sum_{j=1}^n f_j g_j}_X =
    \Bignorm{\sum_{j=1}^n f_j h_j - \sum_{j=1}^n f_j g_j}_X
    \leq \sum_{j=1}^n \|f_j\|_X \|h_j-g_j\|_X
    < 1.
  \end{equation}

  Hence, for every $x\in X$, $f(x) = \sum f_j(x) g_j(x)$ is an invertible element of $A$,
  so that for the function $f=\sum f_jg_j$, which belongs to $\A$, we have $\sing(f)=\es$.
  By the assumption, there is a function $g$ in $\A$ such that
  $\U= fg = \sum f_j (g_jg)$. Now, Theorem \ref{thm:A-is-natural-iff-f1-...-fn} shows that
  $\A$ is natural.
\end{proof}

An application of the above theorem is given in Example \ref{exa:Lip(X,A)}.

\medskip
Let $\A$ be a Banach $A$-valued function algebra. For every point $x\in X$ and
character $\phi\in \mis(A)$ define
\[
  \e_x \diamond \phi:\A\to\C, \quad
  \e_x \diamond \phi(f) = \e_x(\phi\circ f)= \phi(f(x)).
\]
Then $\e_x \diamond \phi$ is a character of $\A$ with
$\ker(\e_x \diamond \phi)=\set{f\in\A:f(x)\in \ker\phi}$, which of course is of
the form \eqref{eqn:natural-form}. Define
\begin{equation}\label{eqn:mapping-J1}
  \J:X\times \mis(A) \to \mis(\A), \quad (x,\phi)\to \e_x \diamond \phi.
\end{equation}

\begin{theorem}
\label{thm:J-1-is-an-embedding-1}
  The mapping $\J$
  is a homeomorphism of $X\times \mis(A)$ onto a compact subset of $\mis(\A)$.
  If $\A$ is natural, then $\mis(\A)$ is homeomorphic to $X\times \mis(A)$.
\end{theorem}

\begin{proof}
  Take $x\in X$, $\phi\in \mis(A)$, and set $\tau_0=\e_{x}\diamond\phi$.
  Let $W$ be a neighbourhood of $\tau_0$ in $\mis(\A)$ of the form
  \[
    W=\set{\tau\in\mis(\A):|\tau(f_i)-\tau_0(f_i)|<\e,\, 1\leq i \leq n},
  \]
  where $f_1,\dotsc,f_n\in \A$. Take
  \begin{align*}
    U &= \set{y\in X:\|f_i(y)-f_i(x)\|<\e/2,\, 1\leq i \leq n}, \\
    V &= \set{\psi\in \mis(A):|\psi(f_i(x))-\phi(f_i(x))|<\e/2,\, 1\leq i \leq n}.
  \end{align*}

  Then $U$ is a neighbourhood of $x$ in $X$ and $V$ is a neighbourhood of $\phi$ in $\mis(A)$,
  so that $U\times V$ is a neighbourhood of $(x,\phi)$ in $X\times \mis(A)$.
  If $(y,\psi)\in U\times V$ then, for every $i$ $(1\leq i \leq n$),
  \begin{align*}
    |\psi(f_i(y))-\phi(f_i(x))|
     & \leq |\psi(f_i(y))-\psi(f_i(x))| + |\psi(f_i(x))-\phi(f_i(x))| \\
     & < \|\psi\|\|f_i(y)-f_i(x)\|+\e/2 \leq \e/2+\e/2=\e.
  \end{align*}

  This shows that $\e_y\diamond \psi\in W$ and thus $\J$ is continuous.
  Finally, if $\A$ is natural then every maximal ideal of $\A$ is of the form
  \eqref{eqn:natural-form} which means that every character $\tau\in\mis(\A)$
  is of the form $\tau=\e_x\diamond \phi$, for some $x\in X$ and $\phi\in \mis(A)$.
  Hence, $\J$ is a surjection and thus a homeomorphism.
\end{proof}

\section{Characters on Vector-valued Function Algebras}
\label{sec:characters}

We turn to a more general case where a vector-valued function algebra
may not be natural. Let $\A$ be a Banach $A$-valued function algebra.
We show that, under certain conditions, the character space $\mis(\A)$ is identical
to $\mis(\fA)\times \mis(A)$, where $\fA=C(X) \cap\A$ is the subalgebra of $\A$
consisting of scalar-valued functions. To this end, we should restrict ourself
to the class of admissible algebras. If $f\in \A$ and $\phi\in\mis(A)$,
it is clear that $\phi\circ f\in C(X)$; it is not, however,  clear whether
the $A$-valued function $(\phi\circ f)\U$ belongs to $\A$. In fact,
\cite[Example 2.4]{Abtahi-vector-valued} shows that it may very well happen that
$(\phi\circ f)\U\notin \A$.

\begin{definition}[\cite{Abtahi-vector-valued}]
  The $A$-valued function algebra $\A$ is called \emph{admissible}و if
  \begin{equation}\label{eqn:admissible-A-valued-FA}
    \set{(\phi\circ f)\U:f\in\A,\, \phi\in\mis(A)}\subset \A.
  \end{equation}
  Note thatو $\A$ is admissible if, and only if, $\phi[\A]\U\subset \A$, for all
  $\phi\in \mis(A)$.
\end{definition}

Admissible vector-valued function algebras exist around in abundant.
Some typical examples are $C(X,A)$, $\Lip(X,A)$, $P(K,A)$, $R(K,A)$, etc.
Tensor products of the form $\fA\tp A$, where $\fA$ is a (Banach) function
algebra on $X$, can be seen as admissible $A$-valued function algebras.
(More details are given in Example \ref{exa:Tensor Products}.)

\medskip
During this section, we assume that $\A$ is admissible and set $\fA=\A\cap C(X)$.
Then $\fA$ is the subalgebra
of $\A$ consisting of all complex functions in $\A$, it forms a complex function
algebra by itself, and $\fA=\phi[\A]$, for all $\phi\in\mis(A)$.
Our aim is to give a description of maximal ideals in $\A$. To begin,
take a character $\phi\in\mis(A)$ and a maximal ideal $\m$ of $\fA$, and set
\begin{equation}\label{eqn:semi-natural-form}
  \M = \set{f\in\A: \phi\circ f\in \m}.
\end{equation}

Then $\M$ is a maximal ideal of $\A$. One way to see this (though it can be seen
directly) is as follows. Take $\psi\in\mis(\fA)$ with $\m=\ker\psi$ and define
\[
  \psi \diamond \phi :\A\to\C, \quad \psi \diamond \phi(f)=\psi(\phi\circ f).
\]
Note that $\psi(\phi\circ f)$ is meaningful since $\phi\circ f\in\fA$.
The functional $\psi \diamond \phi$ is a character of $\A$ with
$\ker(\psi \diamond \phi) = \M$. Hence $\M$ is a maximal ideal of $\A$.
The main question is whether any maximal ideal $\M$ of $\A$ is of
the form \eqref{eqn:semi-natural-form}.

\begin{lemma}
\label{lem:M-is-semi-natural-iff-phi[M] neq fA}
  A maximal ideal $\M$ of $\A$ is of the form \eqref{eqn:semi-natural-form}
  if and only if $\phi[\M]\neq\fA$ for some $\phi\in \mis(A)$.
\end{lemma}

\begin{proof}
  If $\M$ is of the form \eqref{eqn:semi-natural-form} then clearly $\phi[\M]\neq \fA$.
  Conversely, assume that $\phi[\M]\neq \fA$ for some $\phi\in\mis(A)$.
  Then $\phi[\M]$ is an ideal of $\fA$. We show that it is maximal.
  If $g\notin\phi[\M]$, then $g=g\U$ (as an $A$-valued function on $X$)
  does not belong to $\M$. Since $\M$ is maximal
  in $\A$, the ideal generated by $\M\cup \set{g}$ is equal to $\A$. This implies
  that $\U=f+gh$, for some $f\in\M$ and $h\in\A$. Since $\phi\circ g = g$, we get
  $1=\phi\circ f+ g (\phi\circ h)$. This means that the ideal of $\fA$ generated
  by $\phi[\M]\cup\set{g}$ is equal to $\fA$. Thus, $\phi[\M]$ is maximal.
  Set $\m=\phi[\M]$ and $\M_1=\set{f\in\A: \phi\circ f\in \m}$. Then
  $\M\subset \M_1$ and both $\M$ and $\M_1$ are maximal ideals. Hence,
  $\M=\M_1$.
\end{proof}

If $\M=\ker\tau$, for some $\tau\in\mis(\A)$, then $\M$ is of the form
\eqref{eqn:semi-natural-form} if and only if $\tau=\psi\diamond \phi$,
for some $\psi\in\mis(\fA)$ and $\phi\in\mis(A)$.
Let us extend the mapping $\J$ in \eqref{eqn:mapping-J1} to $\mis(\fA)\times\mis(A)$
as follows.
\begin{equation}\label{eqn:mapping-J2}
  \J:\mis(\fA)\times\mis(A)\to \mis(\A), \quad \J(\psi,\phi)= \psi \diamond \phi.
\end{equation}

\noindent
The mapping is injective for if $\psi\diamond\phi=\psi'\diamond\phi'$ then
\begin{align*}
  \phi(a) & =\psi(\phi(a))=\psi'(\phi'(a))=\phi'(a) \quad (a\in A),\\
  \psi(f) & =\psi(\phi(f))=\psi'(\phi'(f))=\psi'(f) \quad (f\in\fA),
\end{align*}
which implies that $\phi=\phi'$ and $\psi=\psi'$. The main question is whether $\J$
is surjective. If $\tau\in\mis(\A)$ then  $\phi=\tau|_A\in \mis(A)$ and
$\psi=\tau|_\fA\in\mis(\fA)$. The question is whether the equality $\tau=\psi\diamond \phi$
holds true; of course, it does hold if $\phi[\M]\neq\fA$.

\begin{theorem}
\label{thm:J-is-homeo-if}
  If the mapping $\J$ in \eqref{eqn:mapping-J2} is a surjection, then
  it is a homeomorphism and, therefore, the character space $\mis(\A)$
  is identical to $\mis(\fA)\times\mis(A)$.
\end{theorem}

\begin{proof}
  Suppose that $\J$ is a surjection (an thus a bijection). Since both the domain and
  the range are compact Hausdorff spaces, it suffices to prove that $\J$ is
  open. Take $\psi_0\in \mis(\fA)$, $\phi_0\in \mis(A)$ and set
  $\tau_0=\J(\psi_0,\phi_0)=\psi_0\diamond \phi_0$. Let $U$ and $V$ be neighborhoods
  of $\psi_0$ and $\phi_0$ of the following form
  \begin{align*}
    U & = \set{\psi\in\mis(\fA): |\psi(f)-\psi_0(f)|<\e_1 \quad (f\in F_1)}, \\
    V & = \set{\phi\in\mis(A): |\phi(a)-\phi_0(a)|<\e_2 \quad (a\in F_2)},
  \end{align*}
  where $F_1$ and $F_2$ are finite sets in $\fA$ and $A$, respectively.
  Take $F=F_1\cup F_2$ as a finite set in $\A$, $\e=\min\set{\e_1,\e_2}$ and
  set
  \[
    W = \set{\tau\in\mis(\A): |\tau(f)-\tau_0(f)|<\e \quad (f\in F)}.
  \]
  Then $W$ is a neighborhood of $\tau_0$ in $\mis(\A)$ and
  $\J(U\times V) \subset W$. Hence $\J$ is open.
\end{proof}

The rest of this section is devoted to investigating conditions under which $\J$ is surjective.

\begin{theorem}[\pp]
\label{thm:tau=psi o phi}
  For a character $\tau\in\mis(\A)$ with $\M=\ker \tau$ and $\phi=\tau|_A$,
  the following are equivalent.
  \begin{enumerate}[\upshape(i)]
    \item \label{item:phi[M] neq fA}
    $\phi[\M]\neq \fA$.

    \item \label{item:M-is-of-semi-natural-form}
    $\M$ is of the form \eqref{eqn:semi-natural-form} with $\m=\phi[\M]$.

    \item \label{item:phi o f = 0 then f in M}
    For every $f\in \A$, if $\phi\circ f=\z$ then $f\in \M$.

    \item \label{item:tau(phi o f)=tau(f)}
    For every $f\in \A$, $\tau(\phi\circ f)=\tau(f)$.

    \item \label{item:f(X)subset M then f in M}
    For every $f\in \A$, if $f(X)\subset\M$ then $f\in\M$.

    \item \label{item:tau=psi diamond phi}
    $\tau=\psi\diamond\phi$, for some $\psi\in\mis(\fA)$.
  \end{enumerate}
\end{theorem}

\begin{proof}
  The equivalence
  $\eqref{item:phi[M] neq fA} \Leftrightarrow \eqref{item:M-is-of-semi-natural-form}$
  is just Lemma \ref{lem:M-is-semi-natural-iff-phi[M] neq fA}.
  The implication $\eqref{item:M-is-of-semi-natural-form} \Rightarrow
  \eqref{item:phi o f = 0 then f in M}$ is clear.
  To see the implication
  $\eqref{item:phi o f = 0 then f in M}\Rightarrow \eqref{item:tau(phi o f)=tau(f)}$,
  let $g=f-(\phi\circ f)\U$. Then $\phi\circ g=\z$
  and thus $g\in\M$ and $\tau(g)=0$. Hence, $\tau(\phi\circ f)=\tau(f)$.
  The implication
  $\eqref{item:tau(phi o f)=tau(f)}\Rightarrow\eqref{item:phi o f = 0 then f in M}$
  is clear.

  To prove $\eqref{item:phi o f = 0 then f in M}\Leftrightarrow
  \eqref{item:f(X)subset M then f in M}$, we note that
  $f(X)\subset\M$ if and only if $\phi\circ f=\z$. In fact,
  $f(X)\subset \M$ means that, for every $x\in X$, the element $f(x)$, as
  a constant function of $X$ into $A$, belongs to $\M$. This, in turn, means that
  $\tau(f(x))=\phi(f(x))=0$, for all $x\in X$, which means that $\phi\circ f=\z$.

  To prove $\eqref{item:phi o f = 0 then f in M}\Rightarrow\eqref{item:tau=psi diamond phi}$,
  first note that $\A$ being admissible implies that
  \[ \fA=\phi[\A]=\set{\phi\circ f:f\in\A}. \]
  Define $\psi:\fA\to\C$ by $\psi(\phi\circ f)=\tau(f)$.
  This is well-defined for if $\phi\circ f=\phi\circ g$ then, by the assumption,
  $f-g\in\M$ which in turn implies that $\tau(f)=\tau(g)$. Obviously,
  $\psi\in\mis(\fA)$ and $\tau=\psi\diamond\phi$.

  Finally, we prove that $\eqref{item:tau=psi diamond phi}\Rightarrow\eqref{item:phi[M] neq fA}$.
  Towards a contradiction, assume that $\phi[\M]=\fA$. Then $\phi\circ f=\U$, for some
  $f\in\M$. Hence $1=\psi(\U)=\psi(\phi\circ f)=\tau(f)=0$
  which is absurd.
\end{proof}

\noindent
\textit{Convention.}
We say that `$\A$ has property \pp' if every $\M\in\mis(\A)$ satisfies one (and hence all)
of conditions in Theorem \ref{thm:tau=psi o phi}. Hence $\A$ has \pp{} if and only if
the mapping $\J$ in \eqref{eqn:mapping-J2} is surjective.

\medskip
Let $\bar\A$ denote the uniform closure of $\A$ in $C(X,A)$. The restriction map
\begin{equation}\label{eqn:restriction-map}
  \mis(\bar\A)\to\mis(\A),\quad \bar\tau\mapsto\bar\tau|_\A,
\end{equation}
is one-to-one and continuous with respect to the Gelfand topology \cite{Honary}.
We write $\mis(\bar\A)=\mis(\A)$ if it is a homeomorphism.

\begin{proposition}
  If $\A$ has \pp{} then $\bar \A$ has \pp. If $\bar \A$ has \pp{} and
  $\norm{\hat f} \leq \norm{f}_X$, for all $f\in \A$, then $\A$ has \pp.
\end{proposition}

\begin{proof}
  Suppose that $\A$ has \pp. Take $\bar\tau\in \mis(\bar\A)$, set $\tau=\bar\tau|_\A$
  and $\phi=\bar\tau|_A=\tau|_A$.  Since $\A$ has \pp,
  by Theorem \ref{thm:tau=psi o phi}~\eqref{item:tau(phi o f)=tau(f)},
  $\tau(\phi\circ f)=\tau(f)$, for all $f\in \A$. Given $f\in\bar \A$,
  there is a sequence $\set{f_n}$ in $\A$ such that $\|f_n-f\|_X\to0$.
  Hence, $\|\phi\circ f_n-\phi\circ f\|_X\to0$, and thus
  \[
    \bar\tau(\phi\circ f)
      = \lim_{n\to\infty} \bar\tau(\phi\circ f_n)
      = \lim_{n\to\infty} \tau(\phi\circ f_n)
      = \lim_{n\to\infty} \tau(f_n)
      = \lim_{n\to\infty} \bar\tau(f_n)
      = \bar\tau(f).
  \]
  Again, by Theorem \ref{thm:tau=psi o phi}~\eqref{item:tau(phi o f)=tau(f)},
  we see that $\bar \A$ has \pp.

  Now, assume that $\bar \A$ has \pp, and $\norm{\hat f} \leq \norm{f}_X$, for all $f\in \A$.
  Take $\tau\in\mis(\A)$ and $\phi=\tau|_A$. Extend $\tau$ to a character
  $\bar\tau:\bar\A\to\C$ (this is possible since $\norm{\hat f} \leq \norm{f}_X$, for all $f\in \A$).
  Note that still we have $\phi=\bar\tau|_A$.
  Since $\bar\A$ satisfies \pp, we have $\bar\tau(\phi\circ f)=\bar\tau(f)$,
  for all $f\in\bar \A$. This implies that $\tau(\phi\circ f)=\tau(f)$, for all $f\in \A$,
  and thus $\A$ has \pp.
\end{proof}

The following is a vector-valued version of Theorem \ref{thm:honary}.

\begin{theorem}
  For an admissible Banach $A$-valued function algebra $\A$ with $\fA=C(X)\cap\A$,
  let $\bar\A$ and $\bar\fA$ be the uniform closures of $\A$ and $\fA$, respectively.
  Consider the following statements:
  \begin{enumerate}[\upshape(i)]
    \item \label{item:restriction-map-cA}
    $\mis(\bar\A)=\mis(\A)$.

    \item \label{item:norm(f)-cA}
    $\norm{\hat f} \leq \norm{f}_X$, for all $f\in \A$.

    \item \label{item:norm(f)-fA}
    $\norm{\hat f} \leq \norm{f}_X$, for all $f\in \fA$.

    \item \label{item:restriction-map-fA}
    $\mis(\bar\fA)=\mis(\fA)$.
  \end{enumerate}
  Then $\eqref{item:restriction-map-cA}\Leftrightarrow \eqref{item:norm(f)-cA}
  \Rightarrow \eqref{item:norm(f)-fA} \Leftrightarrow \eqref{item:restriction-map-fA}$.
  If $\A$ satisfies \pp, then $\eqref{item:norm(f)-fA} \Rightarrow \eqref{item:norm(f)-cA}$.
\end{theorem}

\begin{proof}
  The equivalences $\eqref{item:restriction-map-cA}\Leftrightarrow \eqref{item:norm(f)-cA}$
  and $\eqref{item:norm(f)-fA} \Leftrightarrow \eqref{item:restriction-map-fA}$
  follow from the main theorem in \cite{Honary}. The implication $\eqref{item:norm(f)-cA}
  \Rightarrow \eqref{item:norm(f)-fA}$ is obvious, because $\fA\subset \A$.

  Assume that $\A$ satisfies \pp, and $\norm{\hat f} \leq \norm{f}_X$, for all $f\in \fA$.
  Fix a function $f\in\A$ and take an arbitrary character $\tau\in\mis(\A)$.
  Since $\A$ has \pp, we have $\tau=\psi\diamond\phi$, where $\psi=\tau|_\fA$ and $\phi=\tau|_A$.
  Since $\phi\circ f\in\fA$, we have
  \[
    |\tau(f)|=|\psi(\phi\circ f)| \leq \norm{\widehat{\phi\circ f}}
    \leq \norm{\phi\circ f}_X \leq \norm{f}_X.
  \]
  Hence $\|\hat f\|\leq\|f\|_X$, for all $f\in\A$.
\end{proof}

\section{Examples}
\label{sec:Examples}

To illustrate the results, we devote this section to some examples.

\begin{example}\label{exa:Lip(X,A)}
  Let $(X,\rho)$ be a compact metric space. A function $f : X \to A$
  is called an \emph{$A$-valued Lipschitz function} if
  \begin{equation}
    L(f)=\sup\biggset{\frac{\norm{f(x)-f(y)}}{\rho(x,y)}:x,y\in X,\, x\neq y}<\infty.
  \end{equation}

  The space of $A$-valued Lipschitz functions on $X$ is denoted by $\Lip(X,A)$.
  For any $f\in \Lip(X,A)$, the \emph{Lipschitz norm} of $f$ is defined by
  $\norm{f}_L = \norm{f}_X + L(f)$. This makes $\Lip(X,A)$
  an admissible Banach $A$-valued function algebra on $X$ with
  $\Lip(X)=\Lip(X,A)\cap C(X)$, where $\Lip(X)=\Lip(X,\C)$ is
  the classical complex Lipschitz algebra on $X$.

  The algebra $\Lip(X)$ satisfies all conditions in the Stone-Weierstrass Theorem
  and thus it is dense in $C(X)$. On the other hand, by \cite[Lemma 1]{Hausner},
  $C(X)A$ is dense in $C(X,A)$ and thus $\Lip(X)A$ is dense in $C(X,A)$.
  Since $\Lip(X,A)$ contains $\Lip(X)A$, we see that $\Lip(X,A)$
  is dense in $C(X,A)$.

  It is easy to verify that if $f\in\Lip(X,A)$ and $\sing(f)=\es$,
  then $\U/f\in\Lip(X,A)$. Since $C(X,A)$ is natural, Theorem \ref{thm:if-bar(A)-is-natural}
  now implies that $\Lip(X,A)$ is natural. By Theorem \ref{thm:J-1-is-an-embedding-1},
  $\mis(\Lip(X,A))$ is homeomorphic to $X\times\mis(A)$.
  See also \cite{Esmaeili-Mahyar} and \cite{Honary-Nikou-Sanatpour}.
\end{example}

\begin{example}
  Assume that $A=\C^n$, for some positive integer $n$. Then, for every admissible
  Banach $A$-valued function algebra $\A$ on $X$, the mapping $\J$ in
  \eqref{eqn:mapping-J2} is surjective and thus $\mis(\A)$ is identical
  to $\mis(\fA)\times \mis(\C^n)$.

  To see this,
  we show that $\A$ satisfies condition \eqref{item:phi[M] neq fA} of
  Theorem \ref{thm:tau=psi o phi}. Note that $\mis(\C^n)=\set{\pi_1,\dotsc,\pi_n}$,
  where $\pi_i:\C^n\to\C$ is the projection on $i$-th component. Assume $\M$
  is an ideal in $\A$ and $\U\in\pi_i[\M]$,
  for all $i=1,\dotsc,n$. Hence, for every $i$, there is some $f_i\in\M$ such that
  $\pi_i\circ f_i=\U$. Let $\set{e_1,\dotsc,e_n}$ be the standard basis of $\C^n$.
  Then $e_i$, as a constant function of $X$ into $A$, belongs to $\A$.
  Since $\M$ is an ideal, we have $\U=e_1f_1+\dotsb+e_nf_n\in\M$.
  Hence, $\M=\A$ and $\M$ cannot be maximal.
\end{example}

If $\cX=\set{1,\dotsc,n}$, then $\C^n=C(\cX)$. The above example states that,
given any admissible Banach $C(\cX)$-valued function algebra,
we have $\mis(\A)=\mis(\fA)\times \mis(C(\cX))$.
If $\cX$ is an arbitrary compact Hausdorff space, it is unknown whether the result still
holds for any admissible Banach $C(\cX)$-valued function algebra. But, the following
shows that it does hold for admissible $C(\cX)$-valued uniform algebras.

\begin{example}
  Assume that $A=C(\cX)$, for some compact Hausdorff space $\cX$. Then,
  for every admissible $A$-valued uniform algebra $\A$ on $X$, the mapping $\J$
  in \eqref{eqn:mapping-J2} is surjective,
  and thus $\mis(\A)=\mis(\fA)\times \mis(C(\cX))$.

  To see this, first we show that $\A$ is isometrically isomorphic to $C(\cX,\fA)$.
  Take a function $f\in\A$. Then $f(x)$, for every $x\in X$, is a function in $C(\cX)$.
  Define $\tilde f:\cX\to\fA$ by $\tilde f(\xi)(x) = f(x)(\xi)$.
  In fact, $\tilde f(\xi) = \phi_\xi \circ f$ where $\phi_\xi$ is the evaluation
  character of $A=C(\cX)$ at $\xi$, and, since $\A$ is admissible, $\tilde f(\xi)$ belongs
  to $\fA$. Now, define
  $T:\A \to C(\cX,\fA)$ by $Tf=\tilde f$. It is easily verified that $T$ is an algebra homomorphism, and
  \begin{equation*}
    \|f\|_X = \sup_{x\in X} \|f(x)\| = \sup_{x\in X} \sup_{\xi\in \cX} |f(x)(\xi)|
     = \sup_{\xi\in \cX} \|\tilde f(\xi)\| = \|\tilde f\|_\cX.
  \end{equation*}

  Since the range of $T$ contains all elements of the form
  $g_1h_1+\dotsb+g_nh_n$, where $n\in\N$, $g_i\in C(\cX)$ and $h_i\in\fA$,
  and these functions are dense in $C(\cX,\fA)$, we have $T$ surjective.
  It follows that $T$ is an isometric isomorphism. By \cite[Theorem]{Hausner},
  $\mis(C(\cX,\fA))$ is identical to $\mis(\fA)\times \cX$,
  which means that $\mis(\A)$ is identical to $\mis(\fA)\times \mis(A)$.
\end{example}

\begin{example}[Tensor Products]
\label{exa:Tensor Products}
Let $\fA$ be a Banach function algebra on $X$ and consider the algebraic tensor product
$\fA\tp A$. There exists, by \cite[Theorem 42.6]{BD}, a linear operator $T:\fA\tp A\to \fA A$
such that
\begin{equation}\label{eqn:T:fA tp A to C(X,A)}
  T \Bigprn{\sum_{i=1}^n f_i \tp a_i} = \sum_{i=1}^n f_i a_i.
\end{equation}

The operator $T$ is an algebra isomorphism so that $\fA\tp A$ can be seen as
an admissible $A$-valued function algebra on $X$. We identify every element
$f\in \fA\tp A$ with its image $Tf$ as an $A$-valued function on $X$.
Let $\enorm_\gamma$ be an algebra cross-norm on $\fA\tp A$ so that
the completion $\fA\ptp_\gamma A$ is a Banach algebra.
The mapping $T$ extends to an isometric isomorphism of
$\fA\ptp_\gamma A$ onto a Banach $A$-valued function algebra on $X$.
For example, if $\enorm_\epsilon$ is the injective tensor norm, then
$\fA\ptp_\epsilon A$ is isometrically isomorphic to the uniform closure $\overline{\fA A}$
of $\fA A$ and $\|f\|_\epsilon = \|f\|_X$, for all $f\in \fA\tp A$.

It is proved in \cite{abt-far} that
\begin{enumerate}
  \item $\fA\ptp_\gamma A$ is an admissible Banach $A$-valued function algebra on $X$.

  \item If $f\in \fA \ptp_\gamma A$ and $\phi\in A^*$ then $\phi\circ f\in \fA$
  and $\|\phi\circ f\| \leq \|\phi\|\|f\|_\gamma$.
\end{enumerate}

We now show that every $\tau\in\mis(\fA\ptp_\gamma A)$ is of the form $\tau=\psi\diamond \phi$,
with $\phi=\tau|_A$ and $\psi=\tau|_\fA$. Since $\fA\tp A$ is dense
in $\fA\ptp_\gamma A$, it is enough to show that $\tau=\psi\diamond \phi$
on $\fA\tp A$. First, note that every $f\in \fA\tp A$ can be seen, through
the isomorphism \eqref{eqn:T:fA tp A to C(X,A)}, as
$f=f_1 a_1+\dotsb+f_na_n$. Hence, $\phi \circ f=\phi(a_1)f_1+\dotsb+\phi(a_n)f_n$ and
\begin{align*}
  \tau(f) & = \tau(f_1 a_1+\dotsb+f_na_n) \\
          & = \tau(f_1)\tau(a_1)+\dotsb+\tau(f_n)\tau(a_n) \\
          & = \psi(f_1)\phi(a_1)+\dotsb+\psi(f_n)\phi(a_n) \\
          & = \psi(\phi(a_1)f_1+\dotsb+\phi(a_n)f_n) \\
          & = \psi(\phi\circ f).
\end{align*}

This proves that $\tau=\psi\diamond\phi$ on $\fA\tp A$ and thus
$\tau=\psi\diamond \phi$ on $\fA\ptp_\gamma A$. We conclude that
$\mis(\fA\ptp_\gamma A) = \mis(\fA)\times\mis(A)$. This result, however, can
be derived from the following more general result due to Tomiyama \cite{Tomiyama}.

\begin{theorem}[\cite{Tomiyama}]
   Suppose that $A$ and $B$ are commutative Banach algebras. If $A\ptp_\gamma B$
   is a Banach algebra for a cross-norm $\gamma$, then
   $\mis(A\ptp_\gamma B)$ is homeomorphic to $\mis(A)\times \mis(B)$.
\end{theorem}
\end{example}

\paragraph{\bfseries Acknowledgment}
The author expresses his sincere gratitude to the referee for their careful
reading and suggestions that improved the presentation of this paper.


\begin{thebibliography}{99}

\bibitem{Ab-Ab}
  M. Abel, M. Abtahi,
  \textit{Description of closed maximal ideals in topological algebras of
  continuous vector-valued functions},
  Mediterr. J. Math. \textbf{11} 4 (2014), 1185--1193.

\bibitem{Abtahi-vector-valued}
 M. Abtahi,
 \textit{Vector-valued characters on vector-valued function algebras},
 Banach J. Math. Anal. (to appear)  \texttt{arXiv:1509.09215}.

\bibitem{Abtahi-QM}
 M. Abtahi,
 \textit{Normed algebras of infinitely differentiable Lipschitz functions},
 Quaest. Math. \textbf{35} 2 (2012), 195--202.

\bibitem{abt-far}
 M. Abtahi and S. Farhangi,
 \textit{Vector-valued spectra of Banach algebra valued continuous functions},
 (preprint) \texttt{arXiv:1510.06641}


\bibitem{BD}
 F. F. Bonsall, J. Duncan,
 \textit{Complete Normed Algebras},
 Springer-Verlag, Berlin, Heidelberg, New York, 1973.

\bibitem{Dales}
 {H.G. Dales,}
 \textit{Banach Algebras and Automatic Continuity.}
 London Mathematical Society Monographs, New Series 24, Oxford
 Science Publications, Clarendon Press, Oxford University Press,
 New York, 2000.

\bibitem{Esmaeili-Mahyar}
 K. Esmaeili and H. Mahyar,
 \textit{The  character  spaces  and  \'Silov  boundaries  of  vector-valued lipschitz
 function algebras},
 Indian J. Pure Appl. Math., \textbf{45} 6 (2014), 977--988.

\bibitem{Gamelin-UA}
 T.W. Gamelin,
 \textit{Uniform Algebras},
 Prentice-Hall, Inc., Englewood Cliffs, NJ., 1969.

\bibitem{Hausner}
 A. Hausner,
 \textit{Ideals in a certain Banach algebra},
 Proc. Amer. Math. Soc. \textbf{8} (1957), 246--249.

\bibitem{Honary}
 T.G. Honary,
 \textit{Relations between Banach function algebras and their uniform closures},
 Proc. Amer. Math. Soc., \textbf{109} 2 (1990), 337--342.

\bibitem{Honary-Nikou-Sanatpour}
 T.G. Honary, A. Nikou and A.H. Sanatpour,
 \textit{On the character space of vector-valued Lipschitz algebras},
 Bull. Iranian Math. Soc. \textbf{40} 6 (2014), 1453--1468.

\bibitem{Horvath}
   J. Horv\'ath,
   \textit{Topological Vector Spaces and Distributions}, Vol. I,
   Addison-Wesley, Reading, Mass., 1966.


\bibitem{CBA}
  E. Kaniuth,
  \textit{A Course in Commutative Banach Algebras},
  Graduate Texts in Mathematics, 246 Springer, 2009.

\bibitem{Leibowitz}
 G.M. Leibowitz,
 \textit{Lectures on Complex Function Algebras},
 Scott, Foresman and Company, 1970.


\bibitem{Nikou-Ofarrell}
 A. Nikou and A.G. O'Farrell,
 \textit{Banach algebras of vector-valued functions},
 Glasgow Math. J. \textbf{56} 2 (2014), 419--426.




\bibitem{Stout}
 E.L. Stout,
 \textit{The Theory of Uniform Algebras},
 Bogden \& Quigley, Inc., Tarrytown-on-Hudson, New York, 1971.

\bibitem{Tomiyama}
 J. Tomiyama,
 \textit{Tensor products of commutative Banach algebras},
 T\^{o}hoku Math. J. \textbf{12} (1960), 147--154.

\bibitem{Yood}
 B. Yood,
 \textit{Banach algebras of continuous functions},
 Amer. J. Math., \textbf{73} (1951), 30--42.

\bibitem{Zelazko}
 W. \.Zelazko,
 \textit{Banach Algebras},
 Translated from the Polish by Marcin E. Kuczma.
 Elsevier Publishing Co., Amsterdam-London-New York;
 PWN--Polish Scientific Publishers, Warsaw, 1973.
\end{thebibliography}
\end{document}